\documentclass[10pt]{amsart}
\usepackage{amssymb,amscd,amsmath}
\usepackage{pst-text}
\usepackage[all]{xy}
\newtheorem{thm}{Theorem}[section]
\newtheorem{theorem}[thm]{Theorem}
\newtheorem{cor}[thm]{Corollary}

\newtheorem{prop}[thm]{Proposition}
\newtheorem{lemma}[thm]{Lemma}

\theoremstyle{remark}

\newtheorem{hyp}[thm]{Induction Hypothesis}

\theoremstyle{definition}

\numberwithin{equation}{section}

\newcommand{\Hom}{\mathrm{Hom}}

\newcommand{\Alt}{{\raise 2pt\hbox{$\scriptstyle\bigwedge$}}}

\newcommand{\longlongrightarrow}{\relbar\joinrel\relbar\joinrel\rightarrow}
\newcommand{\longlonglongrightarrow}{\relbar\joinrel\relbar\joinrel\relbar\joinrel\rightarrow}
\newcommand{\THH}{\mathrm{THH}}
\newcommand{\Nil}{\mathrm{Nil}}
\newcommand{\TRM}{\mathcal{ T}_R(M)}
\newcommand{\TpRM}{\mathcal{ T}^{\pi}_R(M)}
\newcommand{\PR}{\mathcal{ P}_R}
\newcommand{\mR}{\mathcal{ M}_R}
\newcommand{\PS}{\mathcal{ P}_S}
\newcommand{\MS}{\mathcal{ M}_S}
\newcommand{\sphere}{{\mathbb{S}}}
\renewcommand{\u}{\underline}
\newcommand{\hocolim}{\mathrm{hocolim}}
\newcommand{\Map}{{\u{\mathrm{Map}}}}

\begin{document}
\title[K-theory of Formal Power Series]
{On the algebraic K-theory of formal power series}

\author{Ayelet Lindenstrauss}
\email{alindens@indiana.edu}
\address{Department of Mathematics\\
    Indiana University \\
    Bloomington, IN 47405\\
    U.S.A.}

\author{Randy McCarthy}
\email{randy@math.uiuc.edu}
\address{Department of Mathematics\\
    University of Illinois at Urbana-Champaign \\
    Urbana, IL 61801\\
    U.S.A.}

\thanks{}

\begin{abstract}
For $R$ a discrete ring, and $M$ a simplicial $R$-bimodule, we let $\TRM$ denote the (derived) tensor algebra of $M$ over $R$, and $\mathcal{ T}^{\pi}_R(M)$ denote the ring of formal (derived) power series in $M$ over $R$.  We define a natural transformation of simplicial $R$-bimodules
$\Phi:\Sigma\tilde K(R;\ \ )\to \tilde K(\mathcal{ T}^{\pi}_R(\ \ ))$
which is closely related to Waldhausen's equivalence 
$\tilde K(\Nil(R;\ \ ))\stackrel{\simeq}{\to}\tilde K(\mathcal{ T}_R(\ \ )).$
We show that $\Phi$ induces an equivalence on any finite stage of the Goodwillie Taylor towers of the functors at any simplicial bimodule.  This is used to show that there is an equivalence of functors
$\Sigma W(R;\ \ )\stackrel{\simeq}{\to} 
{\rm holim}_n \tilde K(\mathcal{ T}_R(\ \ )/I^{n+1}),$
where $W(R;\ \ )$ is what the Goodwillie Taylor tower of $\tilde K(R;\ \ )$ converges to, 
and for connected bimodules, also an equivalence
$\Sigma \tilde K(R;\ \ )\stackrel{\simeq}{\to} \tilde K(\mathcal{ T}_R(\ \ )).$

Read in the opposite direction, the equivalence on the Taylor towers gives us the values that  the finite stages of the Goodwillie Taylor towers of the functor of augmented $R$-algebras $A\mapsto\tilde K(A)$
take on augmented algebras which are of the form $\TRM$ for a connected $R$-bimodule $M$.
\end{abstract}



\maketitle


\section{Introduction}
Throughout this paper, let $R$ be a unital ring, and let $M$ be a simplicial $R$ bi-module.  We will relate the algebraic K-theory of parametrized endomorphisms (the K theory of the category whose objects are pairs $(P,f)$ with $P$ a f.g. projective right $R$-module and $f:P\to P\otimes_R M$  a map of right $R$ modules, with maps being maps of the modules $P$ which induce commutative diagrams) with 
the algebra $\mathcal{ T}^{\pi}_R(M)$ of formal (derived) power series in $M$ 
over $R$ (and in the case of connected bimodules, with the algebraic K-theory of the (derived) tensor algebra $\TRM$ which is weakly equivalent to it).

The idea for the map we use to relate them came from Waldhausen \cite{Wald1}, where he defines an equivalence
$$\Sigma \tilde K(\Nil(R;\ \ ))\stackrel{\simeq}{\to}\tilde K(\mathcal{ T}_R(\ \ ))$$
One can model his $\tilde K(\Nil(R;M ))$ (see also \cite{B}) as the algebraic K-theory of the full subcategory of the category we used to define $K(R;M)$ which consists of modules $P$ and maps
$m:P\to P\otimes_R M$ which are nilpotent, that is, for every $p\in P$ some power $m^{\otimes_Ri}$ vanishes on $p$ (see equation (\ref{tensorsmean}) below for the meaning of $m^{\otimes_Ri}$) .  In these terms, for a nilpotent map $m:P\to P\otimes_R M$, Waldhausen's equivalence sends
$$m\mapsto (1-m)^{-1} = \Sigma_{i=0}^{\infty} m^{\otimes_Ri}$$
where the latter is extended $\TRM$-linearly to be viewed as a map from $P\otimes_R\TRM$ to itself, and the infinite sum makes sense because at every point in the domain of the map, the infinite sum is in fact finite.  This suggests that it would be interesting to look at this map $m\mapsto (1-m)^{-1}$ defined on the full $\tilde K(R;M)$.  Of course, the map would not be able to land in $ \tilde K(\mathcal{ T}_R(M))$ anymore because of the problem of the convergence of the infinite sum, but would land in some sort of (non-commutative, unless $R$ is commutative and $M$ is symmetric with a single generator) localization of it which inverts elements of the form $1-m$ from $P\otimes_R\TRM$ to itself, and the idea would be that one would get a diagram
\begin{equation*}
\xymatrix{
\Sigma \tilde K\ar[d]^{{\rm incl}_*}
(\Nil(R;M )) 
\ar[r]^{{\rm Waldhausen}} &
\Sigma \tilde K\ar[d]^ {}(\mathcal{ T}_R(M))\\
\tilde K(R;M ) 
\ar[r]^{} &
{\rm Appropriate\ \ localization\ \  of\ \ } \tilde K(\mathcal{ T}_R(M))\\}
\end{equation*}
 with the horizontal maps being weak equivalences.  

\medskip
Betley \cite{B} began this program by showing that when $R$ is a field and $M$ a discrete $R$-bimodule, the invariant we call $\tilde K(R;M ) $ is a localization in the sense of Neeman and Ranicki \cite{NR} of $\tilde K(\mathcal{ T}_R(M))$, coming from inverting maps of the form $1-m$ in the category of finitely generated projective $\TRM$ modules.
 
We will instead extend Waldhausen's map by looking at formal power series rather than at the tensor algebras.  Thus for any unital ring $R$, we define a natural transformation
of simplicial $R$-bimodules
$$\Phi:\Sigma\tilde K(R;\ \ )\to \tilde K(\mathcal{ T}^{\pi}_R(\ \ ))$$
and study its behavior on the Taylor towers in the sense of Goodwillie.  The Goodwillie Taylor tower of the parametrized endomorphisms $\Sigma\tilde K(R;M)$ can be described (see \cite{LMcC1}) as follows:  we look at circular derived tensor products of $i$ copies of $M$,
 
 \hbox{ \begin{pspicture}(-3,-2)(1.5,1.5)
    \psset{linestyle=none}
    \pstextpath[c]{\psarcn(0,0){1.2}{0}{10}}
               {$\cdots$}
    \pstextpath[c]{\psarcn(0,0){1.2}{30}{40}}
               {$\hat\otimes_R$}
     \pstextpath[c]{\psarcn(0,0){1.2}{60}{70}}
               {$M$}
    \pstextpath[c]{\psarcn(0,0){1.2}{90}{100}}
               {$\hat\otimes_R$}
     \pstextpath[c]{\psarcn(0,0){1.2}{120}{130}}
               {$M$}
    \pstextpath[c]{\psarcn(0,0){1.2}{150}{160}}
               {$\hat\otimes_R$}
    \pstextpath[c]{\psarcn(0,0){1.2}{180}{190}}
               {$M$}
    \pstextpath[c]{\psarcn(0,0){1.2}{210}{220}}
               {$\hat\otimes_R$}
    \pstextpath[c]{\psarcn(0,0){1.2}{240}{250}}
               {$M$}
    \pstextpath[c]{\psarcn(0,0){1.2}{270}{280}}
               {$\hat\otimes_R$}
    \pstextpath[c]{\psarcn(0,0){1.2}{300}{310}}
               {$M$}
    \pstextpath[c]{\psarcn(0,0){1.2}{330}{340}}
               {$\hat\otimes_R$}
\end{pspicture}}
\noindent
where the cyclic group $C_i$ acts by rotation.  The $n$'th stage of the Goodwillie Taylor tower of 
$\Sigma\tilde K(R;M)$ is the homotopy inverse limit, taken over $i\leq n$, of all the $C_i$ fixedpoints in the circular derived tensor product of $i$ copies of $M$.  When $M=R$, this invariant was introduced in \cite{BHM} and called ${\rm TR}_n(R)$.

We obtain in Theorem \ref{mainone} below that $\Phi$ induces equivalences
\begin{equation}
\label{bothways}
{\Phi_n}:\Sigma W_n(R;\ \ )\stackrel{ \simeq}{\rightarrow}\ \ P_n\tilde K(\mathcal{ T}^{\pi}_R(\ \ ))
\end{equation}
for every $n$. When $M$ is a {\it connected} simplicial bimodule, both of these Goodwillie Taylor towers converge, 
resulting in  Corollary \ref{conn} 
for connected $M$ 
$$\Phi: \Sigma \tilde K(R;M )\stackrel{\simeq}{\to} \tilde K(\mathcal{ T}_R(M )).$$
This is in keeping with the philosophy of extending Waldhausen's map, as explained above, since when $M$ is connected the infinite sum $ \Sigma_{i=0}^{\infty} m^{\otimes_Ri}$ `homotopy converges'.  

For a general bimodule $M$, by comparing $\tilde K(\mathcal{ T}^{\pi}_R(\ \ ))$ to the more tractable functors  $\tilde K(\mathcal{ T}_R(\ \ )/I^{n+1})$, we obtain as Corollary \ref{notconn} the formula
$$\Sigma W(R;M)\stackrel{\simeq}{\to} 
{\rm holim}_n \tilde K(\mathcal{ T}_R(M)/I^{n+1}).$$

\medskip
Another motivation for the result of Corollary \ref{conn} comes from \cite{CCGH}, where Carlsson {\it et al.} show in Theorem 3 that for a simplicial space $X$,
$$A(\Sigma X)\simeq \Sigma \bigvee_n [(\sphere X)^{\wedge n}]_{hC_n}.$$
But by James Milnor splitting, for $X$ connected
$A(\Sigma X)=K(\Sigma^{\infty}(\Omega\Sigma X_+))\simeq K(T_\sphere \sphere X),$
while Tom-dieck splitting (as in \cite{BHM} for the case $R=M$, and more generally as in \cite{I}) gives 
$\Sigma W(\sphere, \sphere X) \simeq\Sigma \bigvee_n [(\sphere X)^{\wedge n}]_{hC_n}.$
In these terms, then, the \cite{CCGH} result could be written for connected $X$ as
$$\Sigma W(\sphere, \sphere X) \simeq K(T_\sphere \sphere X),$$
which is an FSP version (which we do not prove) of our result of Corollary \ref{conn} for $R=\sphere$ and $M=\sphere X$.

\medskip
Read in the opposite direction, if we are interested in understanding the K-theory of augmented simplicial $R$-algebras rather than the K-theory of endomorphisms, the equation (\ref{bothways}) proved in Theorem \ref{mainone} (applied to connected $M$ where $\mathcal{ T}^{\pi}_R(M)\simeq\TRM$) tells us the finite stages of the Goodwillie Taylor tower of the functor $M \mapsto \tilde K(\TRM)$ on simplicial $R$-bimodules (since the stages of the Goodwillie Taylor are determined by their values on connected spaces).  

It would be very interesting to know the Goodwillie Taylor tower of the functor $A\mapsto\tilde K(A)$ on the category of augmented simplicial $R$-algebras.  While that cannot be deduced from our result, it is interesting to note that the Goodwillie Taylor tower of a functor $F$ from augmented simplicial 
$R$-algebras to spectra, when applied to algebras of the form $\TRM$, coincides with the Goodwillie Taylor tower of the functor
$F(\mathcal{ T}_R(\ \ ))$ from  simplicial $R$-bimodules to spectra.  This is because the functor $M\mapsto\TRM$ from $R$-bimodules to augmented $R$-algebras sends the initial and final object $0$  to the initial and final object $R$, coproducts to coproducts, and more generally: co-Cartesian cubes to co-Cartesian cubes. Recall (see  \cite{Calc3}) that  Goodwillie constructs $P_nF(X)=\hocolim_i( T_n^i  F)(X)$, and the iterated maps $T_n$ involve taking homotopy limits of the functor in question over co-Cartesian diagrams of coproducts of $X$ with the initial and final object, so this construction would be the same for $F$ on $R$-algebras and for $F(\mathcal{ T}_R(\ \ ))$ on $R$-bimodules.  Therefore, in this paper we determine the values that the finite Goodwillie Taylor approximations $P_n\tilde K$ take on augmented $R$-algebras which are of the form $\TRM$.

\medskip
In the case $M=R$, the results described above are older.  It has been known
from work by Grayson \cite{Gr} in 1977, at least at the level of homotopy groups, that the K-theory of endomorphisms
$$K(R;R)\simeq \Omega \tilde K((1+xR[x])^{-1} R[x]),$$
where the localization was straightforward because it is done on the level of the underlying commutative ring, and of course $R[x]$ is the same thing as $T_R(R)$.
For $M=R$, the special version
$$TR(R)\simeq{\rm holim}_n \Omega\tilde K(R[x]/(x)^{n+1}$$
of our Corollary \ref{notconn} was proved by Hesselholt \cite{H} in the commutative case, and follows from the work of Betley and Schlichtkrull  in \cite{BS} for general $R$.

\section{Preliminaries}
For $R$ a unital ring  and $M$ an $R$-bimodule we can look at the tensor algebra (with derived tensor products) of $M$ over $R$,
$$\mathcal{ T}_R(M)=R\oplus M\oplus M\!\!\otimes^\wedge_R\!\!M\oplus M\!\!\otimes^\wedge_R\!\!M\!\!\otimes^\wedge_R\!\!M\oplus \cdots.$$
Then $\mathcal{ T}_R(M)$ is an augmented $R$-algebra, and we call its augmentation ideal $I$.
Note that if $M$ is an $R$-bimodule which is flat either as a right $R$-module or as a left $R$-module,
then the tensoring down map $M^{\otimes^\wedge_Rn}\to M^{\otimes _R n}$ is a weak equivalence for every $n$.  This makes $\TRM$ weakly equivalent to the usual tensor algebra for such $M$, which we can denote by $T_R(M)$.

\medskip
We let $\mathcal{P}_R$ denote the category of projective finitely generated right $R$-modules, and $\mathcal{M}_R$ denote the category of finitely generated right $R$-modules.  For an augmented $R$-algebra $A\stackrel{\eta}{\rightarrow}R$ with augmentation ideal $I$ and an element $P\in \mathcal{P}_R$, we will set
$$I_P(A)=\Hom_{\mathcal{M}_R}(P, P\otimes_R I)\cong\ker(\Hom_{\mathcal{M}_R}(P, P\otimes_R A)\stackrel{\eta_*}{\rightarrow}
\Hom_{\mathcal{M}_R}(P, P)).$$
Following the construction in section I.2.5 of \cite{DGMcC}, if we let $\mathbf{1}_P$ denote the identity element in $\Hom(P,P)$, we can view $\mathbf{1}_P+I_P(A)$ as
a subset of 
$\Hom_{\mathcal{M}_R}(P\otimes_RA, P\otimes_R A)$ as follows: for $\alpha\in I_P(A)$,
let
$$(\mathbf{1}_P+\alpha)(p\otimes a) = (p\otimes 1+\alpha(p))a=p\otimes a+\alpha(p)a.$$
Viewed inside $\Hom_{\mathcal{M}_R}(P\otimes_RA, P\otimes_R A)$, we can compose elements of
 $\mathbf{1}_P+I_P(A)$; applying this to $\mathbf{1}_P+\alpha$ and $\mathbf{1}_P+\beta$ will give the composition
\begin{equation}
\label{composition}
P\stackrel{\mathbf{1}_P+\alpha}{\longlongrightarrow}
P\otimes A \stackrel{(\mathbf{1}_P+\beta)\otimes\mathbf{1}_A }{\longlonglongrightarrow}
P\otimes A\otimes A \stackrel{\mathbf{1}_P\otimes\mathrm{mult}_A}{\longlonglongrightarrow}
P\otimes A.
\end{equation}
sending
$$p\mapsto p\otimes 1+\alpha(p)+\beta(p)+\beta(\alpha(p)),$$
where $\beta(\alpha(p))$ is interpreted using $A$-linearity as above.
Note that if $\alpha,\beta$ send $P$ to $P\otimes I$, so does $\alpha+\beta+\beta(\alpha)$, so 
 $\mathbf{1}_P+I_P(A)$ is closed under multiplication.  If there is some reason that  for any $\alpha\in I_P(A)$,  $\mathbf{1}_P +\alpha+\alpha^2+\cdots$ is defined (such as the augmentation ideal being nilpotent or the infinite sum converging for another reason), it is in fact a group.  We look at its classifying space.
 We can by \cite{DGMcC} model the reduced (over $R$) K-theory spectrum $\tilde K(A)$ using Waldhausen's S-construction
 $$\{ n\mapsto \bigvee_{P\in S^{(n)}\mathcal{P}_R} B_.( \mathbf{1}_P+I_P(A)) \}.$$
 We will be looking at the augmented $R$-algebras $\mathcal{ T}_R(M)/I^{n+1}$ and
$$\mathcal{ T}^{\pi}_R(M) = {\rm lim}_{n} \mathcal{ T}_R(M)/I^{n+1}.$$
They both satisfy the condition that  for any $P\in \mathcal{P}_R$ and $\alpha\in I_P(A)$,  $\mathbf{1}_P+\alpha+\alpha^2+\cdots$ is defined in $\Hom_R(P,P\otimes_R A)$.
 
\begin{prop}
\label{analytic}
 The functor $M\mapsto \tilde K(\mathcal{ T}_R(M)/I^{n+1})$, as a functor from simplicial
$R$-bimodules to spectra, 

1) commutes with realizations

2) satisfies the colimit axiom, that is: respects filtered colimits

3) preserves connectivity of maps

4) is -1-analytic

\smallskip
\noindent
so $\tilde K(\mathcal{ T}_R(M)/I^{n+1}) = {\rm holim}_k P_k\tilde K((\mathcal{ T}_R(M)/I^{n+1})$
for all simplicial $R$ bimodules $M$.
\end{prop}

\bigskip\noindent
\begin{proof} Condition 1) follows from chapter III of \cite{DGMcC}. Conditions 2) and 3) follow from the facts
that by direct observation these properties are true for the functor
$M\mapsto\mathcal{ T}_R(M)/I^{n+1}$
and by \cite{Wald} they are true for the algebraic $K$-theory of simplicial rings. 

By taking resolutions if necessary and using the colimit axiom, to show 4)
it suffices to show that the functor of spaces
$$X\mapsto \tilde K(\mathcal{ T}_R(\tilde R[X_+])/I^{n+1})$$
is -1--analytic. To do this we'll follow the process done for the case $n=1$
in \cite{ACTA} (Proposition 3.2) which is essentially a modification of work 
by Goodwillie in \cite{Calc2}. 

Let $\mathcal{ X}$ be a strongly co-Cartesian $S$-cube of spaces. We may assume that
the natural maps are inclusions of sub-simplicial sets. Suppose that the maps
$\mathcal{ X}(\emptyset)\to  \mathcal{ X}(\{s\})$ are $k_s$--connected for each $s\in S$. We wish to show
that the stabilization of the cube of functors
\begin{multline*}
\bigvee_{P\in\mathcal{ P}_R}\!\! B.( \mathbf{1}_P+I_P(\mathcal{ T}_R(\tilde R[X_+])/I^{n+1})) 
\\
\cong \!\!\!\bigvee_{P\in \mathcal{ P}_R}\!\!\! B.( \mathbf{1}_P+{\rm Hom}_{\mathcal{ M}_R}(P,P)\otimes_{\bf Z}\tilde{\bf Z}[{\mathcal{ X}}_+\vee {\mathcal{ X}}_+^2\vee\cdots\vee {\mathcal{ X}}_+^n])
\end{multline*}
is $|S|-1+\Sigma k_s $ Cartesian, since then the functor
$\tilde K(\mathcal{ T}_R(\tilde R[X_+])/I^{n+1}) $ will satisfy $E_{|S|-1}(1-|S|)$ and hence be $-1$ analytic. 
If we show that the cube 
$B.( \mathbf{1}_P+I_P(\mathcal{ T}_R(\tilde R[\mathcal{X}_+])/I^{n+1}))$  is $2(|S|-1) + (\Sigma k_s)$ co-Cartesian for all $P\in \mathcal{ P}_R$, 
then since (homotopy) colimits commute and a $q$-reduced simplicial space of $t$-connected spaces is $(q+t)$--connected, 
$\bigvee_{P\in S^{(q)}_{\bullet}\mathcal{ P}_R}B.( \mathbf{1}_P+I_P(\mathcal{ T}_R(\tilde R[\mathcal{X}_+])/I^{n+1}))$ 
will be $(q+2(|S|-1)+\Sigma k_s)$--co--Cartesian. By taking $\Omega^q$ of these and the limit
with respect to $q$ we will obtain a $2(|S|-1) + \Sigma k_s$-co-Cartesian diagram
of spectra which is equivalent to a $|S|-1 + \Sigma k_s$--Cartesian diagram of spectra (see
\cite{Calc2}, 1.19) and hence the result. 

We will prove that in general, for the $S$--cube $\mathcal{ X}$ 
$$B.( \mathbf{1}_P+I_P(\mathcal{ T}_R(\tilde R[\mathcal{ X}_+])/I^{n+1})) 
\cong B.( \mathbf{1}_P+\Hom_R(P,P)\otimes_{\bf Z}\tilde{\bf Z}[{\mathcal{ X}}_+\vee {\mathcal{ X}}_+^2\vee\cdots\vee {\mathcal{ X}}_+^n])$$
is $2(|S|-1)+\Sigma k_s$ co-Cartesian by induction on $n$. 
We recall that by Theorem 2.6 of \cite{Calc2}, to show an $S$-cube $\mathcal{ Y}$ is
$2(|S|-1)+\Sigma k_s$--co-Cartesian is suffices to show

\begin{hyp}
For each
$T\not = \emptyset$  the  $T$--cube $\partial_{S-T}\mathcal{ Y}$ is
$2(|T|-1) +\Sigma_{t\in T} k_t$-Cartesian.
\end{hyp}
 
In proposition 3.2 of \cite{ACTA} , the case for $n=1$ was done. In particular, the cube
$ B.( \mathbf{1}_P+\Hom_R(P,P)\otimes_{\bf Z}\tilde{\bf Z}[{\mathcal{ X}}_+])$ was shown to satisfy the induction
hypothesis. 
 We have an extension of groups
\begin{multline*}
( \mathbf{1}_P+{\rm Hom}_{\mathcal{ M}_R}(P,P)\otimes_{\bf Z}\tilde{\bf Z}[{\mathcal{ X}}_+^n])\\ \rightarrow
( \mathbf{1}_P+{\rm Hom}_{\mathcal{ M}_R}(P,P)\otimes_{\bf Z}\tilde{\bf Z}[{\mathcal{ X}}_+\vee {\mathcal{ X}}_+^2\vee\cdots\vee {\mathcal{ X}}_+^n])
\\ \stackrel{\pi}{\rightarrow}
( \mathbf{1}_P+{\rm Hom}_{\mathcal{ M}_R}(P,P)\otimes_{\bf Z}\tilde{\bf Z}[{\mathcal{ X}}_+\vee {\mathcal{ X}}_+^2\vee\cdots\vee {\mathcal{ X}}_+^{(n-1)}])
\end{multline*}
Taking the bar construction we obtain a Kan fibration of cubes. By induction, the cube in
the base satisfies the induction hypothesis, and so if the cube in the fiber does also, then
since homotopy pullbacks commute (and these are cubes of connected spaces) the
induction hypothesis will hold for the extension cube. 
The cube 
$B.( \mathbf{1}_P+{\rm Hom}_{\mathcal{ M}_R}(P,P)\otimes_{\bf Z}\tilde{\bf Z}[{\mathcal{ X}}_+^n])$ satisfies
\begin{multline*} 
B.( \mathbf{1}_P+{\rm Hom}_{\mathcal{ M}_R}(P,P)\otimes_{\bf Z}\tilde{\bf Z}[{\mathcal{ X}}_+^n])
\cong
B.( {\rm Hom}_{\mathcal{ M}_R}(P,P)\otimes_{\bf Z}\tilde{\bf Z}[{\mathcal{ X}}_+^n])\\
\cong
 {\rm Hom}_{\mathcal{ M}_R}(P,P)\otimes_{\bf Z} B.(\tilde{\bf Z}[{\mathcal{ X}}_+^n])
 \cong
{\rm Hom}_{\mathcal{ M}_R}(P,P)\otimes_{\bf Z}\tilde{\bf Z}[(\Sigma\mathcal{ X})_+^n].
\end{multline*}
Since
$\Sigma\mathcal{ X}$ is again a strongly co-Cartesian $S$--cube with the
maps $\Sigma\mathcal{ X}(\emptyset)\rightarrow\Sigma\mathcal{ X}(s)$ $k_s+1$
connected for all $s\in S$,
for all $T\not=\emptyset$, $\partial_{S-T}\Sigma\mathcal{ X}$ is a $T$--strongly
co-Cartesian cube with $k_t+1$ connectivity for all $t\in T$ and so the induction
hypothesis is satisfied by
example 4.4 of \cite{Calc2} (for the functor $X\mapsto \tilde Z[X^n_+]$).
\end{proof}

\section{The Main Theorem and its Corollaries}
\begin{theorem} 
\label{mainone}
For $R$ a unital ring,  for every $n$ there is a natural transformation of functors of simplicial $R$-bimodules
$${\Phi_n}:\Sigma W_n(R;\ \ )\stackrel{ \simeq}{\rightarrow}\ \ P_n\tilde K(\mathcal{ T}^{\pi}_R(\ \ ))$$
such that  $\Phi_{n-1}\circ  \Sigma\mathrm{res}_n\simeq p_n\circ \Phi_n$, which is a homotopy equivalence at any simplicial $R$-bimodule.
\end{theorem}

Note that for $M$ which are flat on one side, making 
 the tensoring down map $M^{\otimes^\wedge_Rn}\to M^{\otimes _R n}$ into a weak equivalence for every $n$, we get that $\TRM\simeq T_R(M)$, so  $\TRM/I^{n+1}\simeq T_R(M)/I^{n+1}$ 
 and $\TpRM\simeq T^\pi_R(M)$.
 Therefore by \cite{Wald}, $\tilde K(\TRM)\simeq \tilde K(T_R(M))$ and $\tilde K(\TpRM)\simeq \tilde K(T^{\pi}_R(M))$.

\medskip
For the definition and properties of the $W_n$,  see \cite{LMcC1}.  It is the inverse limit over all $i\leq n$ of the $C_i$ fixedpoints in the cyclic derived tensor over $R$ of $i$ copies of $M$ which was described in the introduction.  The map  $\mathrm{res}_n$ comes from restriction of categories over which limits are taken from $\{1,2,\ldots,n\}$ to  $\{1,2,\ldots,n-1\}$; $p_n$ are the connecting maps of the Goodwillie tower of the functor.
 
In  \cite{LMcC1}, it is shown that $W_n(R;\ \ )= P_n \tilde K(R;\ \ )$, from which it follows that 
$\Sigma W_n(R;\ \ )= P_n(\Sigma \tilde K(R;\ \ ))$.  What Theorem \ref{mainone} in fact shows is that there
exists a natural transformation inducing an equivalence between Goodwillie Taylor towers of the functors $\Sigma \tilde K(R;\ \ )$ and $\tilde K(\mathcal{ T}^{\pi}_R(\ \ ))$.  
Moreover,  we can draw the following

\begin{cor}
\label{notconn}
For any unital ring $R$, there is a natural equivalence of functors of  simplicial $R$-bimodules
$$\Sigma W(R;\ \ )\to 
{\rm holim}_n \tilde K(\mathcal{ T}_R(\ \ )/I^{n+1}) .$$
\end{cor}

 \begin{proof}{\it (Of Corollary \ref{notconn}, given Theorem \ref{mainone})}
We have, by  Theorem \ref{mainone} above, that for any $R$-bimodule $M$,
\begin{multline}
\label{above}
\Sigma W(R;M)\stackrel{{\rm def}}{=}
{\rm holim}_k \Sigma W_k(R;M)\simeq {\rm holim}_k P_k\tilde K(\mathcal{ T}^\pi _R(M))
\\   \stackrel{{\rm def}}{=}
 {\rm holim}_k P_k\tilde K({\rm holim}_n \mathcal{ T}_R(M )/I^{n+1}).
\end{multline}
But the Taylor tower of the functor $K({\rm holim}_n \mathcal{ T}_R(\ \  )/I^{n+1})$ can be determined by applying it to $M$ connected, where the map ${\rm holim}_n \mathcal{ T}_R(M )/I^{n+1} \to
 \mathcal{ T}_R(M )/I^{n_0+1}$ can be as connected as we want it to be, and $\tilde K$ preserves connectivity of maps, so
 $$P_k\tilde K({\rm holim}_n \mathcal{ T}_R(M )/I^{n+1})
 =  {\rm holim}_n P_k\tilde K(\mathcal{ T}_R(M )/I^{n+1}),$$
 which we can plug into 
equation (\ref{above}) to get
\begin{multline*}
\Sigma W(R;M)\simeq
 {\rm holim}_k {\rm holim}_n P_k\tilde K(\mathcal{ T}_R(M )/I^{n+1})
 \\
 \simeq  {\rm holim}_n {\rm holim}_k P_k\tilde K(\mathcal{ T}_R(M )/I^{n+1})
 \simeq  {\rm holim}_n \tilde K(\mathcal{ T}_R(M )/I^{n+1}),
\end{multline*}
where the last equality is the convergence of the Taylor tower for $\tilde K(\mathcal{ T}_R(\ \  )/I^{n+1})$
fro Proposition \ref{analytic} 
above.
\end{proof}
\medskip
\begin{cor}
\label{conn}
If $R$ is a unital ring, there is a natural equivalence of functors of connected simplicial $R$-bimodules
$$\Phi:\Sigma \tilde K(R;\ \ )\to \tilde K(\mathcal{ T}_R(\ \ )).$$
\end{cor}

 \begin{proof}{\it (Of Corollary \ref{conn}, given Theorem \ref{mainone})}
The natural transformation $\Phi$ is that introduced in the beginning of the proof of Theorem \ref{mainone}, which induces the $\Phi_n$'s.   The point is that for connected $M$, both the Taylor tower of $\Sigma \tilde K(R;\ \ )$ converges to 
$\Sigma \tilde K(R;M )$ (since that of $ \tilde K(R;\ \ )$ converges to $ \tilde K(R;M )$), and the Taylor tower of 
$\tilde K(\mathcal{ T}^{\pi}_R(\ \ ))$ converges to $\tilde K(\mathcal{ T}^{\pi}_R(M))$.  
Moreover, for connected $M$, the map 
$$\mathcal{ T}_R(M)\rightarrow \mathcal{ T}_R^{\pi}(M)$$
 is
an equivalence.   The fact that this map is an equivalence for connected $M$ shows that the map is order $n$ for all $n$
and hence $P_nK(\mathcal{ T}_R (\ \ ))\stackrel{\simeq}{\rightarrow} P_nK(\mathcal{ T}^{\pi} _R(\ \ ))$ for all $n$, so the convergence of the Taylor tower for $\tilde K^{\pi} (\mathcal{ T}_R(\ \ ))$ for connected $M$ is in fact the convergence of the Taylor tower for $\tilde K(\mathcal{ T}_R(\ \ ))$ for such $M$.

The convergence of the Taylor tower for $ \tilde K(R;\ \ )$  follows from Theorem 9.2 in \cite{LMcC1}, which shows it for the special case of $M=\tilde N[X_.]$ for $X_.$ connected, $N$ discrete.
To go from that to the general case of $M$ connected, observe that $\tilde K(R;\ \ )\simeq\tilde K(R\ltimes \Sigma\ \ )$ commutes with realizations by \cite{Wald}.  The finite stages of the Taylor tower
$W_n(R;\ \ )$ commute with realizations as finite inverse limits of the $U_a(R;\ \ )^{C_a}$ which are directly seen to commute with realizations, but for $M$ connected the map $W(R;M)\to W_n(R;M)$ is $n$-connected, that is, it can be as connected as we like by taking $n$ large enough, so $W(R;\ \  )$ commutes with realizations for connected bimodules.
We want to show that the map to the Taylor tower $ \tilde K(R;\ \ )\to W(R;\ \  )$ is an equivalence for any connected bimodule, and both sides commute with realizations for such bimodules.  But any connected bimodule is homotopy equivalent to the realization of a bisimplicial set, assigning to each $n$ a simplicial set of the form covered by Theorem  9.2 in \cite{LMcC1}: Given a general connected simplicial $R$-bimodule $M$, we can first represent it by a reduced one (that has only a single $0$-simplex, the basepoint) by looking at the sub-simplicial bimodule $M_0$ consisting of all the simplices in $M$ all of whose vertices are at the basepoint.  The inclusion $M_0\hookrightarrow M$ is an equivalence on $\pi_0$ by assumption, and on all higher homotopy groups by the definition of the homotopy groups of a simplicial abelian group.  Then, replace $M_0$ by its $R\otimes R^{\rm op}$-free simplicial resolution
$$\widetilde {R\otimes R^{\rm op}} [M_0] 
\ \ \lower 3pt \vbox{\hbox{$\leftarrow$}\vskip -6pt\hbox{$\leftarrow$}}\ \ 
\widetilde {R\otimes R^{\rm op}} [\widetilde {R\otimes R^{\rm op}} [M_0]]
 \ \ \lower 6pt \vbox{\hbox{$\leftarrow$}\vskip -6pt\hbox{$\leftarrow$}\vskip -6pt\hbox{$\leftarrow$}}\ \ 
\widetilde {R\otimes R^{\rm op}} [\widetilde {R\otimes R^{\rm op}} [\widetilde {R\otimes R^{\rm op}} [M_0]]]\cdots
$$
in which each stage is of the form $\tilde N[X_.]$ for $N=R\otimes R^{\rm op}$ discrete and a connected simplicial $X_.$.

The convergence of the Taylor tower for $\tilde K(\mathcal{ T}^{\pi}_R(\ \ ))$ for $M$ connected is due to the following facts: by Proposition \ref{analytic} above, the Taylor towers converge for 
$\tilde K(\mathcal{ T}_R(M)/I^{n+1})$, that is $\tilde K(\mathcal{ T}_R(M)/I^{n+1}) \simeq {\rm holim}_k P_k\tilde K(\mathcal{ T}_R(M)/I^{n+1})$.   
The map $\mathcal{ T}^{\pi}_R(M) \to \mathcal{ T}_R(M)/I^{n+1}$ is as connected as we want it to be for $n$ large enough, and since $\tilde K(\ \ )$ preserves connectivity of maps by \cite{Wald}, we get that 
\begin{multline*}
\tilde K(\mathcal{ T}^{\pi}_R(M)) \simeq 
{\rm holim}_n \tilde K(\mathcal{ T}_R(M)/I^{n+1}) 
\simeq
{\rm holim}_n {\rm holim}_k P_k\tilde K(\mathcal{ T}_R(M)/I^{n+1})
\\ \simeq
{\rm holim}_k {\rm holim}_n P_k\tilde K(\mathcal{ T}_R(M)/I^{n+1})
\simeq 
{\rm holim}_k P_k \tilde K(\mathcal{ T}^{\pi}_R(M)).
\end{multline*}
\end{proof}

\bigskip\noindent
\begin{proof}{\it (Of Theorem \ref{mainone})}
The augmentation ideal $I$ for $\mathcal{ T}_R^{\pi}(M)\rightarrow R$ for any $M$ is such that $1+I$ is
contained in the units of $\mathcal{ T}_R^{\pi}(M)$ and hence the fiber of the
map $K(\mathcal{ T}_R^{\pi}(M))\rightarrow K(R)$ can by section I.2.5 of \cite{DGMcC} be modeled as the stabilization in Waldhausen's S-construction 
of the functor
$$\bigvee_{P\in\mathcal{ P}_R}B.( \mathbf{1}_P+I_P(\mathcal{ T}_R^{\pi}(M))).$$
As before,  $I_P(\mathcal{ T}_R^{\pi}(M)) = {\rm Hom}_{\mathcal{ M}_R}(P,P\otimes I)$ is considered as the
ideal given by the kernel of the ring map
${\rm Hom}_{\mathcal{ M}_R}(P,P \otimes_R\mathcal{ T}_R^{\pi}(M))\rightarrow {\rm Hom}_{\mathcal{ M}_R}(P,P)$.

We define a natural transformation 
$$\phi:\Sigma\tilde K(R;M)\rightarrow \tilde K(\mathcal{ T}^{\pi}_R(M))$$
as the stablization of the natural transformation between the model of $\tilde K(R;M)$ as the stabilization of
$\bigvee_{P\in\mathcal{ P}_R}\Hom_{\mathcal{ M}_R}(P, P\otimes_R M)$ and the above model of
$ \tilde K(\mathcal{ T}^{\pi}_R(M))$ which for a map $m\in \Hom_{\mathcal{ M}_R}(P, P\otimes_R M)$ sends
$$m\mapsto (1-m)^{-1} = \Sigma_{i=0}^{\infty} m^{\otimes_Ri}$$
The point is that $0$-simplices in $\tilde K(R;M)$ become $1$-simplices in its suspension; each such $1$-simplex which comes from the $0$-simplex $m$ is mapped to a $1$-simplex in the classifying space corresponding to the element
$$ \Sigma_{i=0}^{\infty} m^{\otimes_Ri} \in  B_1( \mathbf{1}_P+I_P(\mathcal{ T}_R^{\pi}(M)))=\mathbf{1}_P+I_P(\mathcal{ T}_R^{\pi}(M)).$$
  Note that, for example, the notation $m^{\otimes _R 2}$ means the composition
\begin{equation}
\label{tensorsmean}
P\stackrel{m}{\rightarrow} P\otimes _R M
 \stackrel{m\otimes \mathbf{1}_M}{\longlongrightarrow} 
 P\otimes _R M  \otimes _R M 
\end{equation}
 and is therefore also in $I_P(\mathcal{ T}_R^{\pi}).$

\medskip
What we want to show is that this natural transformation $\phi$ induces an equivalence of the Goodwillie Taylor towers at the basepoint ${*}$, and these are determined by what they do on sufficiently connected spaces.  Thus, the theorem will follow once we show that $\phi$ induces an equivalence
after one suspension. We would like to establish the result using analytic continuation
as in \cite{Calc3}.  In order to do this we first must observe that $\tilde K(\mathcal{ T}^{\pi}_R(B.\ \ ))$ commutes with realizations,
has the limit axiom and is -1--analytic.  These are all true because the fact that  $\mathcal{ T}^{\pi}_R(B.\ \ )\rightarrow \mathcal{ T}_R(B.\ \ )/I^{n+1}$
is $n$ connected for all $n$ implies $\tilde K(\mathcal{ T}^{\pi}_R(B.\ \ ))\rightarrow \tilde K(\mathcal{ T}_R(B.\ \ )/I^{n+1})$
is $n$ connected for all $n$ and these results hold for $K(\mathcal{ T}_R(B.\ \ )/I^{n+1})$ for all $n$ by Proposition \ref{analytic} above.  

Thus, we fix our $M$, which we may assume to be connected, and are interested in the fibers of 
$\Sigma\tilde K(R;M\oplus N)\rightarrow \Sigma\tilde K(R;M)$ and of
$\tilde K(\mathcal{ T}^{\pi}_R(M\oplus N))\rightarrow \tilde K(\mathcal{ T}^{\pi}_R(M))$ in 
a $2n$ range when $N$ is $n$--connected.  Since we assume that $M$ is connected,
$ \mathcal{ T}_R(M)\stackrel{\simeq}{\to} \mathcal{ T}^{\pi}_R (M)$ and 
$ \mathcal{ T}_R(M\oplus N)\stackrel{\simeq}{\to} \mathcal{ T}^{\pi}_R (M\oplus N)$
.

For the fiber of $\Sigma\tilde K(R;M\oplus N)\rightarrow \Sigma\tilde K(R;M)$, we can describe it  using  \cite{LMcC1} which shows that for connected bimodules $ \tilde K(R,\  )\simeq W(R;\ )$ (that is, for connected bimodules the Taylor tower converges to $\tilde K(R;\ )$) together with the splitting of Theorem 2.2 in \cite{LMcC2} for $W(R;\ )$.  We get that for $M,N$ connected, 
\begin{multline*}
\tilde K(R;M\oplus N) 
\\ =\bigvee_{a=1}^\infty\   \ \bigvee_{\{f:\ \{1,\ldots,a\}\to\{ M,N\}\ \mathrm{non\ periodic}\} /C_a}\tilde K(R;f(1)\otimes^\wedge_R\cdots\otimes^\wedge_Rf(a)),
\end{multline*}
where $C_a$ acts on functions $ \{1,\ldots,a\}\to\{ M,N\}$ by permuting  $ \{1,\ldots,a\}$ cyclically before applying the function, and a function $f$ is considered periodic if for some $b| a$, the value of $f(i)$ is determined by the remainder of $i$ when divided by $b$, that is: when if we write the values of $f$ as a word of length $a$ in $M$ and $N$, that word is a word of length $b$ repeated $a/b$ times.  It follows from the discussion there that the maps $M\hookrightarrow M\oplus N\stackrel{p_1}{\rightarrow} M$ embed $\tilde K(R;M)$ as
the direct summand corresponding to the function from the set of one element $1\mapsto M$, so the homotopy fiber of the projection map consists of all the other summands.

The fiber of $\tilde K(\mathcal{ T}_R(M\oplus N))\rightarrow \tilde K(\mathcal{ T}_R(M))$ is exactly the algebraic K-theory  of $\mathcal{ T}_R(M\oplus N)$ reduced over $\TRM$.  We can compare this reduced algebraic K-theory  to that of another ring:  Note that, by sending any terms with more than one tensored entry in $N$ to the basepoint, we have a $2n$-connected multiplicative map
$$ \mathcal{ T}_R(M\oplus N) \stackrel{\Psi}{\to}\TRM\ltimes(\TRM\otimes^{\wedge}_R
N \otimes^{\wedge}_R\TRM).$$

\medskip
We can put all this together in a commutative diagram
\begin{equation}
\label{maineqn}
\xymatrix{
\Sigma\ar[d]^{\phi}
\tilde K(R;M)
\ar[r]^{} &
\Sigma\ar[d]^{\phi}
\tilde K(R;M\oplus N)
\ar[r]^{} &
 \Sigma\ar[d]^{\alpha}\tilde K(R;N)\vee\bigvee \Sigma\tilde K(R;\otimes_if(i))\\
\tilde K \ar[d]^{=} (\TRM)
\ar[r]^{} &
\tilde K \ar[d]^{\Psi_*}
( \mathcal{ T}_R(M\oplus N))
\ar[r]^{} &
\tilde K_{\TRM}
(\mathcal{ T}_R(M\oplus N))\ar[d]^{\beta}\\ 
\tilde K(\TRM)
\ar[r]^{} &
\tilde K
("\TRM\ltimes N")
\ar[r]^{} &\tilde K_{\TRM}
("\TRM\ltimes N")\\}
\end{equation}
where
\begin{multline*}
\bigvee \Sigma\tilde K(R;\otimes_if(i))
\\ =\bigvee_{a=2}^\infty\   \ \bigvee_{\{f:\ \{1,\ldots,a\}\to\{ M,N\}\ \mathrm{non\ periodic}\} /C_a} \Sigma\tilde K(R;f(1)\otimes^\wedge_R\cdots\otimes^\wedge_Rf(a))
\end{multline*}
and
$$"\TRM\ltimes N"= \TRM\ltimes(\TRM\otimes^{\wedge}_R
N \otimes^{\wedge}_R\TRM).$$
The left column maps to the center column by maps induced by the obvious inclusions.  Since the inclusions are all inclusions of retracts, the spectra in the center column all split as the product of the spectrum on their left and the spectrum on their right.  

\medskip
Our goal is to show that when $N$ is $n$-connected, $\alpha$ is $2n$-connected in equation (\ref{maineqn}).  This would mean that $\phi$ induces an equivalence of the Goodwillie differentials at $M$.  Since $\Psi$ is $2n$-connected, $\Psi_*$ and therefore also $\beta$ are $2n$-connected as well.  So our strategy will be to show that $\beta\circ\alpha$ is $2n$-connected, and deduce from that that $\alpha$ is. 

It is plausible that $\beta\circ\alpha$ is $2n$-connected, since we will now see that its target and source have the same homotopy type in these dimensions.  In the next sections, we will see that $\beta\circ\alpha$ actually induces a $2n$-equivalence.

We can map
\begin{multline}
\label{oneway}
\Sigma\tilde K(R;N)\vee\bigvee_{a=2}^\infty\   \  \bigvee_{\{f :\  \{1,\ldots,a\} \to \{ M,N\}\  \mathrm{non\ periodic}\} /C_a} 
\Sigma\tilde K(R;f(1)\otimes^\wedge_R\cdots\otimes^\wedge_Rf(a))\\
\stackrel{{2n}}{\to}
\bigvee_{a=0}^\infty \Sigma\tilde K(R;M^{\otimes^\wedge_R a}\otimes^\wedge_R N)
\stackrel{{2n}}{\to}
\bigvee_{a=0}^\infty \Sigma\tilde \THH(R;M^{\otimes^\wedge_R a}\otimes^\wedge_R N)\\
\simeq\THH(R;\TRM\otimes^\wedge_R\Sigma N)
\end{multline}
where the first map collapses all terms corresponding to $f$'s which hit $N$ more than once, and since reduced K-theory sends $2n$-connected bimodules to $2n$-connected spectra, it is $2n$-connected; the second map is $2n$-connected by \cite{DMcC1}, and the last map is an equivalence by the linearity of $\THH$ in the bimodule variable and since $M$ is connected.

By \cite{DMcC1},
\begin{multline*}
\tilde K_{\TRM}
( \TRM\ltimes \TRM\otimes^\wedge_R
N \otimes^\wedge_R\TRM\simeq K(\TRM;\Sigma \TRM\otimes^\wedge_R
N \otimes^\wedge_R\TRM)\\
\stackrel{2n}{\to}
\THH(\TRM;\Sigma \TRM\otimes^\wedge_R
N \otimes^\wedge_R\TRM)
\end{multline*}
and by Lemma \ref{scalars} below, 
there is a homotopy equivalence
$$\THH(R;\TRM\otimes^\wedge_R\Sigma N)
\stackrel{\simeq}{\to} 
\THH(\TRM;\TRM\otimes^\wedge_R\Sigma N\otimes^\wedge_R\TRM),$$
the same spectrum we ended up with in equation (\ref{oneway}).

\section{Checking that the Equivalence is Induced by the Correct Map}
This section is dedicated to finishing the proof of Theorem \ref{mainone} by tracing the maps in (\ref{maineqn}) to establish that $\beta\circ\alpha$ in fact induces a $2n$-equivalence for $N$ $n$-connected.  We will first need some lemmas, which will all be proven in the last section of the paper.

\begin{lemma}
\label{longwayround}
Model $\tilde K(R; M\oplus N)$ by Waldhausen's S-construction as the stabilization of 
$$ \bigvee_{P\in S^{(n)}\PR}\Hom_R
(P, P\otimes_R(M\oplus N))\cong\bigvee_{P\in S^{(n)}\PR}
(\Hom_R(P, P\otimes M)\oplus \Hom_R(P, P\otimes N));$$
model, similarly,
$$\tilde K_{\TRM}(\TRM\ltimes(\TRM\otimes^{\wedge}_R
N \otimes^{\wedge}_R\TRM))\simeq
\tilde K(\TRM; B_.(\TRM\otimes^{\wedge}_R
N \otimes^{\wedge}_R\TRM))
$$ 
(this is the homotopy equivalence of \cite{DMcC1})
as the stabilization of 
$$\bigvee_{Q\in S^{(n)}\mathcal{P}_{\TRM}} \Hom_{\TRM}(Q, Q\otimes_{\TRM}B_.(\TRM\otimes^{\wedge}_R
N \otimes^{\wedge}_R\TRM))
.$$
Then if we start at the middle of the top row of diagram (\ref{maineqn}), follow the maps $\phi$ and $\Psi_*$ down and then the map which goes right, the resulting map
\begin{multline*}
\tilde K(R; M\oplus N)\to \tilde K_{\TRM}(\TRM\ltimes(\TRM\otimes^{\wedge}_R
N \otimes^{\wedge}_R\TRM)) \\
\simeq
\tilde K(\TRM; B_.(\TRM\otimes^{\wedge}_R
N \otimes^{\wedge}_R\TRM))
\end{multline*}
is induced by sending the suspension $\Sigma(m,n)$ of each $(m,n)\in \Hom_R(P, P\otimes M)\oplus \Hom_R(P, P\otimes N)$ to
\begin{multline*}
(\mathbf{1}_{P\otimes _R \TRM}-m)^{-1}\otimes n
\\
\in \Hom_{\TRM}(P\otimes_R\TRM, P\otimes_R \TRM \otimes_{\TRM}\TRM\otimes^{\wedge}_R
N \otimes^{\wedge}_R\TRM)
\\
= \Hom_{\TRM}(P\otimes_R\TRM, P\otimes_R \TRM \otimes_{\TRM}B_1(\TRM\otimes^{\wedge}_R
N \otimes^{\wedge}_R\TRM))
\end{multline*}
in the summand corresponding to $Q=P\otimes_R\TRM$.
\end{lemma}

\begin{lemma}
\label{scalars}
Let $R$ be a ring spectrum, and $S$ an $R$-algebra.  Let $X$ be an $S-R$ bimodule.  Then
there is a homotopy equivalence
$$\THH(R;X)\stackrel{\simeq}{\to} \THH(S;X\otimes^\wedge_R S).$$
\end{lemma}

\begin{lemma}
\label{inducedmap}
Let $R$ be a simplicial ring, $S$ a simplicial $R$-algebra., and  $X$ a simplicial $S-R$ bimodule.  Then
if we construct $\THH(R;M)$ for an $R$-bimodule $M$ via the Waldhausen S-construction
$$\{ n\mapsto \oplus_{P\in S^{(n)}\PR}\Hom_{S^{(n)}\mR}(P, P\otimes_R M)\},$$
the isomorphism 
$$\THH(R;X)\simeq\THH(R;S\otimes^\wedge_R X)$$
of Lemma \ref{scalars} for the associated Eilenberg Mac Lane spectra is induced by the map
$$
\oplus_{P\in S^{(n)}\PR}\!\Hom_{S^{(n)}\mR}(P, P\otimes_R X)
\!\!\to\!
\oplus_{Q\in S^{(n)}\PS}\!\Hom_{S^{(n)}\MS}(Q, Q\otimes_S (X\otimes^\wedge_R S))$$
sending
$$\alpha:\ P\to P\otimes_R X$$
to
$$\alpha\otimes^\wedge_R \mathbf{1}_S: P\otimes^ \wedge_RS=P\otimes_R S\to P\otimes_R S\otimes_SX\otimes^\wedge_R S\simeq P\otimes_R X\otimes^\wedge_R S.$$
\end{lemma}

\bigskip
Using these lemmas, we will be able to complete our proof of Theorem \ref{mainone}.
 The \cite{DMcC1} map
\begin{multline}
\label{linearization}
\tilde K(\TRM; B.(\TRM\otimes^{\wedge}_R
N \otimes^{\wedge}_R\TRM)) 
\\ \stackrel{2n}{\to}
\THH(\TRM; B.(\TRM\otimes^{\wedge}_R
N \otimes^{\wedge}_R\TRM))
\end{multline}
is obtained simply by passing from $\bigvee$ to $\bigoplus$,
\begin{multline*}
\bigvee_{Q\in\mathcal{R}_{\TRM}}\Hom_{\TRM}(Q, Q\otimes_{\TRM} B.( \TRM\otimes^{\wedge}_R
N \otimes^{\wedge}_R\TRM))
\\
\to
\bigoplus_{Q\in\mathcal{R}_{\TRM}}\Hom_{\TRM}(Q, Q\otimes_{\TRM} B.( \TRM\otimes^{\wedge}_R
N \otimes^{\wedge}_R\TRM)).
\end{multline*}
So by Lemma \ref{longwayround} above, the composition of $\Psi_*\circ\phi$ with the \cite{DMcC1} linearization map (\ref{linearization}) sends the loop represented by the $1$-simplex $\Sigma(m,n)$ in the $P$ summand to something homotopic to the loop represented by the $1$-simplex $(\mathbf{1}_{P\otimes _R \TRM}-m)^{-1}\otimes n$ in the $P\otimes _R \TRM$ summand.
 
\medskip 
 Alternatively, if we look at the map onto the cofiber on the top row of equation (\ref{maineqn}) and then collapse to a point the terms with more than one $N$,
 \begin{multline}
 \label{otherway}
 \tilde K(R;M\oplus N)
 \\
 \to\tilde K (R:N)\vee
 \bigvee_{a=2}^\infty\   \ \bigvee_{\{f:\ \{1,\ldots,a\}\to\{ M,N\}\ \mathrm{non\ periodic}\} /C_a} \tilde K(R;f(1)\otimes^\wedge_R\cdots\otimes^\wedge_Rf(a))
 \\
 \stackrel{2n}{\to}
 \bigvee_{a=0}^\infty \tilde K(R;M^{\otimes^\wedge_R a}\otimes^\wedge_R N)
\simeq
 \prod_{a=0}^\infty \tilde K(R;M^{\otimes^\wedge_R a}\otimes^\wedge_R N)
 \end{multline}
 it is induced by the stabilization of the product of the maps
\begin{multline*}
\bigvee_{P\in\PR}
(\Hom_R(P, P\otimes M)\oplus \Hom_R(P, P\otimes N)) \to
\bigvee_{P\in\PR}
\Hom_R(P, P\otimes M^{\otimes^\wedge_R a}\otimes^\wedge_R N)\\
\end{multline*}
sending
$$(m,n)\mapsto m^{\otimes a}\otimes n.
$$
To see this, we use the fact that $M$ is connected and $N$ is $n$-connected for $n$ which we may assume to be at least $1$, and therefore $M\oplus N$ is connected as well.
Then
\begin{multline}
\label{decomp}
\tilde K(R;M\oplus N)
\stackrel{\simeq}{\to}
W(R;M\oplus N)
\\
\simeq 
\prod_{a=1}^\infty\   \ \bigvee_{\{f:\ \{1,\ldots,a\}\to\{ M,N\}\ \mathrm{non\ periodic}\} /C_a} \!\!W(R;f(1)\otimes^\wedge_R\cdots\otimes^\wedge_Rf(a))
\\
\stackrel{2n}{\to}
W(R;M)\times\prod_{a=0}^\infty
 W(R; M^{\otimes^\wedge_R a}\otimes^\wedge_R N)
 \stackrel{\simeq}{\leftarrow} 
 \tilde K(R;M)\times\prod_{a=0}^\infty
 \tilde K(R; M^{\otimes^\wedge_R a}\otimes^\wedge_R N)
 \end{multline}
 Here the first and last equivalences are by Theorem 9.2 in \cite{LMcC1}, the second one is the splitting of Theorem 2.2 in \cite{LMcC2}, and the third map is $2n$-connected because if $f$ hits $N$ more than once, $f(1)\otimes^\wedge_R\cdots\otimes^\wedge_Rf(a)$ and therefore also $W$ of $R$ with coefficients in it are $2n$-connected.  
 
 We are, of course, quotienting this whole picture out by $\tilde K(R;M)$.  By following the decomposition of Theorem 2.2 in \cite{LMcC2} on the $0$-dimensional part (as we did before), we see that
 $$(m,n)\in \Hom_R(P, P\otimes M)\oplus \Hom_R(P, P\otimes N)$$
 in $\tilde K(R;M\oplus N)$ in the beginning of equation (\ref{decomp}) lands in the same place 
 in
 $W(R;M)\times\prod_{a=0}^\infty
 W(R; M^{\otimes^\wedge_R a}\otimes^\wedge_R N)$
 as $\{m\}\times\prod_{a=0}^\infty \{ m^{\otimes a}\otimes n\}$ in $\tilde K(R;M)\times\prod_{a=0}^\infty
 \tilde K(R; M^{\otimes^\wedge_R a}\otimes^\wedge_R N)$ on the right.
 Since the spectra we are looking at increase in connectivity, we know that their sum $\bigvee$ is homotopy equivalent to their product.
 
 In (\ref{maineqn}), we are using the suspension of the map of (\ref{otherway}),
 $$\Sigma \tilde K(R;M\oplus N) \to \Sigma \bigvee_{a=0}^\infty
 \tilde K(R; M^{\otimes^\wedge_R a}\otimes^\wedge_R N)$$ and want to compose it with
 \begin{multline*}
\Sigma \bigvee_{a=0}^\infty
 \tilde K(R; M^{\otimes^\wedge_R a}\otimes^\wedge_R N) \stackrel{2n+1}{\longrightarrow}
 \Sigma \bigvee_{a=0}^\infty
 \THH(R; M^{\otimes^\wedge_R a}\otimes^\wedge_R N)\\
 \simeq \THH(R;\Sigma(\bigoplus_{a=0}^\infty M^{\otimes^\wedge_R a}\otimes^\wedge_R N)
 \simeq\THH(R;\Sigma \TRM\otimes^\wedge_R N),
 \end{multline*}
 where the first map is that of \cite{DMcC1}, and the second uses the linearity of $\THH$ in the bimodule coordinate.
 
 So the suspension of $(m,n)\in  \Hom_R(P, P\otimes M)\oplus \Hom_R(P, P\otimes N)$ lands in the $1$-simplex corresponding to
 $$\sum_{a=0}^\infty m^{\otimes a}\otimes n= (\mathbf{1}_{P}-m)^{-1}\otimes n\in\Hom_R(P, P\otimes _R \TRM\otimes^\wedge_R N).$$
 When we went the $\Psi_*\circ \phi $ route, instead of getting
 $$(\mathbf{1}_{P}-m)^{-1}\otimes n\in\Hom_R(P, P\otimes _R \TRM\otimes^\wedge_R N)\subset\THH(R;  \TRM\otimes^\wedge_R N)$$
 we got
 \begin{multline*}
 (\mathbf{1}_{P\otimes _R \TRM}-m)^{-1}\otimes n
 \\
 \in\Hom_{\TRM}(P\otimes_R\TRM, P\otimes _R \TRM
  \otimes_{\TRM}
  \TRM\otimes^\wedge_R N\otimes^\wedge_R\TRM)
  \\
  \subset\THH(\TRM;  \TRM\otimes^\wedge_R N\otimes^\wedge_R\TRM).
  \end{multline*}
  But by Lemmas \ref{scalars} and \ref{inducedmap} above, that is exactly what we need to assure ourselves that up to homotopy, the map $\beta\circ\alpha$ is the $2n$-equivalence we were after. 
\end{proof}

\section{Proofs of the Technical Lemmas}

\noindent{\bf Lemma \ref{longwayround}} \begin{proof}
 The discussion will be done over $\PR$, i.e. in the first stage of the Waldhausen S-construction, but can be carried over to $S^{(n)}\PR$ for any $n$.

The map
$$\Sigma\tilde K(R;M\oplus N)
\stackrel{\phi}{\to}
\tilde K(\mathcal{ T}_R(M\oplus N))$$
was induced by stabilizing the map 
\begin{equation*}
\xymatrix{
\Sigma
\ar[d]^{\simeq} 
 \bigvee_{P\in\PR}\Hom_R
(P, P\otimes_R(M\oplus N))
\ar[r]^{} &
\bigvee_{P\in\PR}B.(\mathbf{1}_P+I_P(\mathcal{ T}_R(M\oplus N)))
\\
\Sigma
\bigvee_{P\in\PR}
(\Hom_R(P, P\otimes M)\oplus \Hom_R(P, P\otimes N))&
\\}
\end{equation*}
sending $\Sigma(m,n)$ to the $1$-simplex
$$\mathbf{1}_P+ (m+n)+(m+n)^{\otimes 2} + (m+n)^{\otimes 3} +\cdots\in B_1(\mathbf{1}_P+I_P(\mathcal{ T}_R(M\oplus N))).$$
Now $\Psi_*: \ \tilde K(\mathcal{ T}_R(M\oplus N))\to\tilde K(
 \TRM\ltimes(\TRM\otimes^{\wedge}_R
N \otimes^{\wedge}_R\TRM))$ is induced by
\begin{multline*}
\Psi_*: \ \bigvee_{P\in\PR}B.(\mathbf{1}_P+I_P(\mathcal{ T}_R(M\oplus N)))
\\ \to
\bigvee_{P\in\PR}B.(\mathbf{1}_P+I_P( \TRM\ltimes(\TRM\otimes^{\wedge}_R
N \otimes^{\wedge}_R\TRM)))
\end{multline*}
so the original $1$-simplex $\Sigma(m,n) $ will be further sent to the $1$-simplex
\begin{multline*}
\mathbf{1}_P+ \sum_{i=1}^\infty
m^{\otimes i} + \sum_{j,k=0}^\infty
m^{\otimes j} \otimes n\otimes m^{\otimes k}\\
=(\mathbf{1}_P-m)^{-1}
+ (\mathbf{1}_P-m)^{-1}\otimes n\otimes (\mathbf{1}_P-m)^{-1}.
\end{multline*}

\medskip
Until now, we have looked at K-theory of $R$ algebras reduced over $R$, which we can emphasize by writing $\tilde K_R(\ )$.  To get $\tilde K_{\TRM}(\TRM\ltimes(\TRM\otimes^{\wedge}_R
N \otimes^{\wedge}_R\TRM))$, we look at the homotopy fiber of the map
$$\tilde K_R(\TRM\ltimes(\TRM\otimes^{\wedge}_R
N \otimes^{\wedge}_R\TRM))\rightarrow \tilde K_R(\TRM).$$ 
Since $ \tilde K_R(\TRM)$ is a direct summand in $\tilde K_R(\TRM\ltimes(\TRM\otimes^{\wedge}_R
N \otimes^{\wedge}_R\TRM))$, so is the homotopy fiber of the map, and when we map $\tilde K_R(\TRM\ltimes(\TRM\otimes^{\wedge}_R
N \otimes^{\wedge}_R\TRM))$ down to the homotopy fiber, the image of  $ \tilde K_R(\TRM)$
is identified to a point.  So on the pre-stabilized version,  $\bigvee_{P\in\PR}B.(\mathbf{1}_P+I_P( \TRM\ltimes(\TRM\otimes^{\wedge}_R
N \otimes^{\wedge}_R\TRM)))$, we know that  anything coming from  $ \tilde K_R(\TRM)$, that is
$ \bigvee_{P\in\PR}B.(\mathbf{1}_P+I_P(\mathcal{ T}_R(M)))$, has to be identified to a point in the homotopy fiber.

But note that if we have a subgroup $H\subset G$ and collapse the subspace $BH\subset BG$ to a point, in the quotient space $BG/BH$, for any $h\in H$, $g\in G$ the $1$-simplex corresponding to $hg$ is homotopic via the $2$-simplex $(h,g)$ (one of whose edges has been collapsed to a point) to the $1$-simplex corresponding to $g$.

Since 
\begin{multline*}
\tilde K_R(\TRM\ltimes(\TRM\otimes^{\wedge}_R
N \otimes^{\wedge}_R\TRM)) \\
\simeq
\tilde K_R(\TRM)\times \tilde K_{\TRM}(\TRM\ltimes(\TRM\otimes^{\wedge}_R
N \otimes^{\wedge}_R\TRM)),
\end{multline*}
when we pass from $\tilde K_R$ to $\tilde K_{\TRM}$, we are identifying $\tilde K_R(\TRM)$ to a point, and so at each level of the stabilization, the image of our original $1$-simplex $\Sigma(m,n)$ will be homotopic (rel endpoints) to the image of the $1$-simplex
$$\mathbf{1}_P+ (\mathbf{1}_P-m)^{-1}\otimes n \in B_1(\mathbf{1}_P+I_P( \TRM\ltimes(\TRM\otimes^{\wedge}_R
N \otimes^{\wedge}_R\TRM))).$$

\medskip
Now when $\tilde K_{\TRM}(\TRM\ltimes(\TRM\otimes^{\wedge}_R
N \otimes^{\wedge}_R\TRM)$ is represented as the stabilization of
$$\bigvee_{Q\in\mathcal{P}_{\TRM}}B.(\mathbf{1}_Q+I_Q( \TRM\ltimes(\TRM\otimes^{\wedge}_R
N \otimes^{\wedge}_R\TRM))),
$$
(note that the augmentation ideal here refers now to augmentation over $\TRM$), then $\mathbf{1}_P+ (\mathbf{1}_P-m)^{-1}\otimes n$ 
should be viewed there not as an $R$-linear map
$$P\to P\otimes_R (\TRM\ltimes(\TRM\otimes^{\wedge}_R
N \otimes^{\wedge}_R\TRM))$$
but as its $\TRM$-linear extension to
\begin{multline*}
P\otimes_R\TRM\to (P\otimes_R\TRM)\otimes_{\TRM}
(\TRM\ltimes(\TRM\otimes^{\wedge}_R
N \otimes^{\wedge}_R\TRM))
\\
\cong P\otimes_R (\TRM\ltimes(\TRM\otimes^{\wedge}_R
N \otimes^{\wedge}_R\TRM)).
\end{multline*}
(Extending maps $\TRM$-linearly
gives an isomorphism$$\Hom_R(P,  P\otimes_R S)\leftrightarrow
\Hom_{\TRM}(P\otimes_R \TRM, (P\otimes_R\TRM)\otimes_{\TRM} S).\ )$$

So now our original simplex $\Sigma(m,n)$ maps to the $1$-simplex
\begin{multline*}
\mathbf{1}_{P\otimes _R \TRM}+ (\mathbf{1}_{P\otimes _R \TRM}-m)^{-1}\otimes n \\
\in B_1(\mathbf{1}_{P\otimes _R \TRM}+I_{P\otimes _R \TRM}( \TRM\ltimes(\TRM\otimes^{\wedge}_R
N \otimes^{\wedge}_R\TRM)))
\end{multline*}
in the 
$\bigvee_{Q\in\mathcal{R}_{\TRM}}B.(\mathbf{1}_Q+I_Q( \TRM\ltimes(\TRM\otimes^{\wedge}_R
N \otimes^{\wedge}_R\TRM)))$ model of 
$\tilde K_{\TRM}(\TRM\ltimes(\TRM\otimes^{\wedge}_R
N \otimes^{\wedge}_R\TRM)$ which admits a homotopy equivalence
\begin{multline*}
\tilde K(\TRM; B.(\TRM\otimes^{\wedge}_R
N \otimes^{\wedge}_R\TRM))
\\
\stackrel{\simeq}{\to}
\tilde K_{\TRM}(\TRM\ltimes(\TRM\otimes^{\wedge}_R
N \otimes^{\wedge}_R\TRM))
\end{multline*}
described in section 4 of \cite{DMcC1}. Following the description there, this homotopy equivalence is the stabilization of a given map
\begin{multline*}
\bigvee_{Q\in\mathcal{P}_{\TRM}}\Hom_{\TRM}(Q, Q\otimes_{\TRM} B.( \TRM\otimes^{\wedge}_R
N \otimes^{\wedge}_R\TRM))
\\
\to \bigvee_{Q\in\mathcal{P}_{\TRM}}B.(\mathbf{1}_Q+I_Q( \TRM\ltimes(\TRM\otimes^{\wedge}_R
N \otimes^{\wedge}_R\TRM)))
\end{multline*}
which sends the $1$-simplex corresponding to 
\begin{multline*}
\alpha:\ Q\to Q\otimes_{\TRM} ( \TRM\otimes^{\wedge}_R
N \otimes^{\wedge}_R\TRM)
\\
=Q\otimes_{\TRM} B_1( \TRM\otimes^{\wedge}_R
N \otimes^{\wedge}_R\TRM)
\end{multline*}
to
\begin{multline*}
\mathbf{1}_Q+\alpha\in \mathbf{1}_Q+I_Q( \TRM\ltimes(\TRM\otimes^{\wedge}_R
N \otimes^{\wedge}_R\TRM)
\\
=B_1(\mathbf{1}_Q+I_Q( \TRM\ltimes(\TRM\otimes^{\wedge}_R
N \otimes^{\wedge}_R\TRM)).
\end{multline*}
Therefore the loop corresponding to $\mathbf{1}_{P\otimes _R \TRM}+ (\mathbf{1}_{P\otimes _R \TRM}-m)^{-1}\otimes n $ in
$\tilde K_{\TRM}(\TRM\ltimes(\TRM\otimes^{\wedge}_R
N \otimes^{\wedge}_R\TRM))$ comes from
$$(\mathbf{1}_{P\otimes _R \TRM}-m)^{-1}\otimes n:\ P\otimes _R \TRM\to \TRM\otimes^{\wedge}_R N\otimes^\wedge_R\TRM.$$
\end{proof}

\noindent{\bf Lemma \ref{scalars}} \begin{proof}
We look at the bisimplicial spectrum
$$(p,q)\mapsto S^{\wedge p}\wedge X\wedge R^{\wedge q}\wedge S$$
with the usual Hochschild-type face and degeneracy maps in both simplicial dimensions.
Realizing first in the $p$-direction, we get that the realization of this bisimplicial set is 
$$\THH(R;S\otimes^\wedge_S X)\simeq\THH(R;X);$$
realizing first in the $q$-direction, we get that the realization is $$\THH(S;X\otimes^\wedge_R S).$$
\end{proof}

\noindent{\bf Lemma \ref{inducedmap}} \begin{proof}
We will use the methods of \cite{DMcC2}: we can model the product
$\u S^{\wedge p}\wedge\u  X\wedge \u R^{\wedge q}\wedge \u S$ of the Eilenberg Mac Lane spectra associated to the simplicial rings and modules by
\begin{multline}
\hocolim_{\stackrel{\to}{\u X}} \Map 
(S^{\sqcup \u X}, 
\bigvee_{\u A }
\Hom_S(Q_1, Q_0)[S^{X^1_0}]
\wedge\cdots\wedge
\Hom_S(Q_p, Q_{p-1})[S^{X^1_{p-1}}]   \\
\wedge\Hom_R(P_0, Q_p\otimes_S X)[S^{X^1_p}]\wedge
\Hom_R(P_1, P_0)[S^{X^2_0}]
\wedge\cdots  \\
\wedge
\Hom_R(P_q, P_{q-1})[S^{X^2_{q-1}}]
\wedge\Hom_S(Q_0, P_q\otimes_R S)[S^{X^2_q}])
\end{multline}
where $\u X=(X^1_0,\ldots,X^1_p,X^2_0,\ldots,X^2_q)$ is a collection of finite sets and where 
$$\u A=(Q_0,\ldots, Q_p, P_0,\ldots, P_q),\ \ Q_i\in\PS,\ P_i\in\PR.$$
Boundary maps in this model come from the composition of maps, smashed with identity maps of a bimodule as needed, and the smashing together of spheres.

For the elements we need to represent, we can take $p=q=0$ and $X^i_j=\emptyset\ \forall i,j$,
and look at elements in
$$\Hom_R(P,Q\otimes_S X)\wedge \Hom_S(Q,P\otimes_R S)$$
for $P\in\PR$ and $Q\in\PS$.

Given a $P\in\PR$ and an $R$-linear map $\alpha:\ P\to P\otimes_R X$ (where the $S-R$-bimodule $X$ is viewed as a left $R$-module through the unit map $R\to S$), we look at $Q=P\otimes_R S\in \PS$.
Since
$$Q\otimes_S X=(P\otimes_R S)\otimes_S X\cong P\otimes_R X,$$
$\alpha$ can be viewed as an element of $\Hom_R(P,Q\otimes_S X)$.  Consider
$$(\alpha,\mathbf{1}_Q)\in \Hom_R(P,Q\otimes_S X)\wedge \Hom_S(Q,P\otimes_R S).$$
If we map $\u S^{\wedge p}\wedge\u  X\wedge \u R^{\wedge q}\wedge \u S\to \THH(R; X\otimes^\wedge_S S)$, then $(\alpha,\mathbf{1}_Q)$ will be identified with the composition
$$P\stackrel{\alpha}{\to} Q\otimes_S X\stackrel{ \mathbf{1}_Q \otimes_S  \mathbf{1}_X}{\longlongrightarrow} (P\otimes_R S)\otimes_S X\cong P\otimes_R X
$$
that is, with $\alpha\in\Hom_R(P,P\otimes_R X)$.  But if we map 
$\u S^{\wedge p}\wedge\u  X\wedge \u R^{\wedge q}\wedge \u S\to \THH(S; X\otimes^\wedge_R S)$,
$(\alpha,\mathbf{1}_Q)$ will be identified with the composition
$$Q\stackrel{\mathbf{1}_Q}{\to} P\otimes_R S 
\stackrel{ \alpha \otimes_R \mathbf{1}_S}{\longlongrightarrow}
(Q\otimes_S X)\otimes_R S$$
which is the map we called 
$$\alpha\otimes^\wedge_R \mathbf{1}_S\in\Hom_S(P\otimes^\wedge_R S, P\otimes^\wedge_R X\otimes^\wedge_RS).$$  
This argument holds  for any stage $n$ in the Waldhausen S-construction.
\end{proof}

\bigskip

\begin{thebibliography}{}

\bibitem[B]{B} S. Betley, Algebraic K-theory of parametrized endomorphisms, {\it K-Theory} {\bf 36} (2005) no. 3--4, 291--303.

\bibitem[BHM]{BHM} M. B\"okstedt, W.-C. Hsiang, I. Madsen,  The cyclotomic trace and algebraic K-theory of spaces, {\it Invent. Math}. {\bf 111} (1993), no. 3, 465--539.

\bibitem[BS]{BS} S. Betley, C. Schlichtkrull, The cyclotomic trace and curves on K-theory, {\it Topology} {\bf 44} (2005), no. 4, 845--874.

\bibitem[CCGH]{CCGH} G. E. Carlsson, R. L. Cohen, T. Goodwillie, W.-C. Hsiang, The free loop space and the algebraic K-theory of spaces, {\it  K-Theory}  {\bf 1}  (1987),  no. 1, 53--82. 

\bibitem[DGMcC]{DGMcC} B. Dundas, T. Goodwillie, R. McCarthy, 
{\it The Local structure of algebraic K-theory},
preprint of book.

\bibitem[DMcC1]{DMcC1} B. Dundas, R. McCarthy, Stable K-theory and
topological Hoch\-schild homology, {\it Ann. of Math.} {\bf 140} (1994), 685--701.

\bibitem[DMcC2]{DMcC2} B. Dundas, R. McCarthy, Topological Hochschild homology of ring functors and exact categories, {\it J. Pure Appl. Algebra } {\bf 109} (1996), no. 3, 231--294.

\bibitem[G2]{Calc2} T. Goodwillie, Calculus II, Analytic functors, {\it K-Theory} {\bf 5} (1991/92), no. 4, 295--332.

\bibitem[G3]{Calc3} T. Goodwillie, Calculus III, Taylor series, {\it Geometry and Topology} {\bf 7} (2003), 645--711.

\bibitem[Gr]{Gr} D. Grayson, K-theory of endomorphisms, {\it J. Algebra} {\bf 48} (1977), 439--446.

\bibitem[H]{H} L. Hesselholt, On the p-typical curves in Quillen's K-theory, {\it Acta Math.} {\bf 177} (1996), 1--53.

\bibitem[I]{I} Y. Iwachita, {\it The Lefschetz-Reidemeister Trace in Algebraic K-theory}, Ph. D. thesis, UIUC, 1999.

\bibitem[LMcC1]{LMcC1} A. Lindenstrauss, R. McCarthy,  On the Taylor tower of relative K-theory, preprint.

\bibitem[LMcC2]{LMcC2} A. Lindenstrauss, R. McCarthy,  The algebraic K-theory of extensions of a ring by direct sums of itself, {\it Indiana Univ. Math. J.} {\bf 57} (2008), no. 2, 577--626.

\bibitem[McC]{ACTA} R. McCarthy, Relative algebraic K--theory and topological cyclic homology,
{\it Acta Math.} {\bf 179} (1997), no. 2, 197--222.   

\bibitem[NR]{NR} A. Neeman, A. Ranicki,  Noncommutative localisation in algebraic K-theory. I, {\it  Geom. Topol}  {\bf 8}  (2004), 1385--1425.

\bibitem[W1]{Wald1} F. Waldhausen,  Algebraic K-theory of generalized free products. I, II, {\it  Ann. of Math. (2)}  {\bf 108}  (1978), no. 1, 135--204.

\bibitem[W2]{Wald} F. Waldhausen, Algebraic K-theory of spaces, in {\it Algebraic and Geometric Topology (Rutgers 1983)}, 318--419, {\it Lecture Notes in Math.} {\bf 1126}, Springer-Verlag Berlin-New York, 1985.

 
\end{thebibliography}
\end{document}